\documentclass[runningheads]{llncs}
\usepackage{booktabs}
\usepackage{graphicx}
\def\C{\mathbb{C}}
\def\R{\mathbb{R}}
\def\H{\mathbb{H}}
\def\O{\mathbb{O}}

\def\bv{\mathbf{v}}

\def\bx{\mathbf{x}}

\def\cC{\mathcal{C}}
\def\cF{\mathcal{F}}
\def\cH{\mathcal{H}}
\def\cI{\mathcal{I}}
\def\cO{\mathcal{O}}

\def\cF{\mathcal{F}}
\def\cT{\mathcal{T}}

\def\fA{\mathfrak{A}}

\def\fK{\mathfrak{K}}
\def\fM{\mathfrak{M}}

\usepackage{mathtools}
\usepackage{xcolor}
\usepackage{amssymb}
\usepackage{amsmath}

\def\e{\varepsilon}

\begin{document}
\title{Algebraic properties of the information geometry's fourth Frobenius manifold}
%

%
\author{N. Combe \and
P. Combe\and H. Nencka\thanks{This research was supported by the Max Planck Society's Minerva grant. The authors express their gratitude towards MPI MiS for excellent working conditions.}}
\authorrunning{N. Combe, P. Combe, H. Nencka}
%

\institute{Max Planck Institute for Maths in the Sciences, Inselstrasse 22 ,04103 Leipzig, Germany\\ noemie.combe@mis.mpg.de}
%
\maketitle              
\begin{abstract}

Recently, it has been shown that within the statistical manifold, related to exponential families, there exists a submanifold having a Frobenius manifold structure. This appears as the fourth class of Frobenius manifolds. It has a structure of a  projective manifold over a rank two Frobenius algebra $\fA$, being the algebra of paracomplex numbers and generated by $1, \e$ such that $\e^2=1$. 
This last result is a key step towards an algebraization of the results concerning the manifold of probability distributions and thus offers a new perspective on it. In this paper, we prove that the fourth Frobenius manifold is decomposed into a pair of symmetric totally geodesic pseudo-Riemannian submanifolds, each of which correspond to a module over an ideal of $\fA$. This pair of ideals are orthogonal idempotents. The symmetry is obtained under the Peirce mirror.

\end{abstract}
\section{Introduction}
The structure of Frobenius manifolds and its weakened version (the $F$-manifolds) developed in~\cite{HeMa99}, was discovered in the process of axiomatisation of Topological Field Theory, including Mirror Conjecture.

A tightly related notion is the one of Frobenius algebras, which  is deeply connected to (1+1)-Topological Field Theories. If one has a 2-dimensional topological field theory, then naturally there arises the structure of a commutative Frobenius algebra over  a field $k$ on the corresponding vector space and reciprocally.

According to Dubrovin~\cite{Du} the main component of a Frobenius structure on a manifold $M$ is a (super)commutative, associative and bilinear over constants multiplication $\circ: TM \otimes TM \to TM$ on its tangent sheaf $TM$. Additional structures are listed below:

\begin{itemize}
\item A subsheaf of flat vector fields $TM^f \subset TM$ consisting of tangent vectors flat in a certain affine structure.
\item A metric (nondegenerate symmetric quadratic form) $g : S^2(TM ) \to O_M$ .
\item An identity $e.$
\item An Euler vector field $E$.
\end{itemize} 

The three famous classes of Frobenius manifolds include quantum cohomology (topological sigma-models), unfolding spaces of singularities (Saito’s theory, Landau-Ginzburg models), and Barannikov--Kontsevich construction starting with the Dolbeault complex of a Calabi--Yau manifold and conjecturally producing the $B$-side of the mirror conjecture in arbitrary dimension (see \cite{MaF}).

As has been shown in \cite{CoMa} this list is not exhaustive. Moreover, it is proved in \cite{CoMa} that there exists a submanifold of the finite dimensional manifold of probability distributions, related to exponential families, which fulfils the axioms of a Frobenius manifold (and also of an $F$-manifold, the weaker version of Frobenius manifolds). The classical setting for statistical manifolds, under consideration in this paper is given in the Appendix.
 This paper is devoted to exploring the algebraic structure of this fourth Frobenius manifold, relying on the methods and approach developed in~\cite{CoMa}.

\smallskip 

We now present the main statement of this article. 
\subsection*{Main theorem}
Let $\fA$ be a rank two Frobenius algebra over the field of real numbers, being a group ring of a finite dimension 2 group over the field of real numbers and generated by \[\langle 1,\ \e\, |\, \e^{2}=1\rangle.\] 

\begin{theorem}[Main theorem]
 The fourth $F$-manifold has the following properties:
\begin{enumerate}
\item It is identified to a projective manifold, defined over the rank two Frobenius algebra $\fA$.
\item The fourth $F$-manifold decomposes into a pair of isomorphic totally geodesic submanifolds\footnote{Let $M$ be a manifold and $g$ a Riemannian metric.  A submanifold $N$ of a Riemannian manifold $(M,g)$ is called {\it totally geodesic} if any geodesic on the submanifold $N$ with its induced Riemannian metric $g$ is also a geodesic on the Riemannian manifold $(M,g)$. For more details, we refer to \cite{KoNo}.
}, respectively equipped with a flat connection $\nabla$ and $\nabla^*$. 
\item The pair of pseudo-Riemannian submanifolds of the fourth $F$-manifold are symmetric to each other, with respect to a Peirce reflection.
\end{enumerate}
\end{theorem}

\begin{remark}The second property corresponds to the realisations of the module over the maximal ideal of $\fA$ and its annihilator.\end{remark}
In the following corollary we present a dictionary between the algebra and the geometry. 
\vfill\eject
\begin{corollary}
The results are summarised in the following table. 
 \begin{table}[h!]
  \begin{center}
\label{tab:table1}
    \begin{tabular}{|c|c|c|} 
    \toprule
  &    \textbf{Algebraic properties}  & \quad \textbf{Geometric interpretation}\\
      \hline
      
         Algebra & Generated by $1$ and $\e$ & linear space over $\fA$/\\
       
         $\fA$&  where $\e^{2}=1$ & manifold over $\fA$\\\hline
     Idempotents  & $\fA$ has a pair of orthogonal &  \\  
   & idempotents: &  Real interpretation  \\
   &  &   in the real linear space:  \\
   &$e_{-}= \frac{1-\e}{2},\quad e_{+}=\frac{1+\e}{2}$& pair of $m$-eigenspaces \\
 && $E_+^m$ and $E_-^m$  \\  \hline
$\fA$-module  &   & Tangent space $T_{\theta}$  at point $P_{\theta}$   \\ 
  && is identified to a  $\fA$-module  \\ \hline
 Real interpretation  & Pair of eigenspaces &$\fA$-manifold decomposes \\
 
in the real linear space &$E^{m}_{+} = \{v\in E^{2m}_\fA \, |\, \fK v=\e v\}$ & into a direct sum \\

&$E^{m}_{-} = \{v\in E^{2m}_\fA \, |\, \fK v=-\e v\}$ &  of submanifolds being totally geodesic.\\  \hline
Connections & & The submanifolds have \\ 
  & & flat connections \\\hline
Peirce decomposition & Peirce decomposition & The submanifolds are symmetric
 \\  
 & and involution & to each other with respect to Peirce mirror.\\ \bottomrule
 \end{tabular}
  \caption{Algebraic and geometric dictionary}
  \end{center}
\end{table}
\end{corollary}

\subsection*{Method of proof}
The fourth $F$-manifold enters the context of Norden--Shirokov's theory on manifolds \cite{No52,No54,No58,Sh54,Sh02}, and which is deeply intertwined with Weyl's geometry. Our approach somewhat differs from Norden's perspective, since we do not rely on almost complex structures. However, we remedy to this using Shirokov's theory \cite{Sh54,Sh02}, which allows to work on modules over finite, unital, commutative algebras and handle algebras containing idempotents. Furthermore, this allows us to use Peirce's decomposition theorem for algebras and, in the case of existence of idempotents Peirce reflections \cite{MC04}.  

\section{The fourth Frobenius manifolds}
For the reader's convenience, we recall in this section the definition of Frobenius manifolds and related to notions.
\subsection{Frobenius manifolds}
Let us start with the following data:
\begin{equation}(M; \circ:TM \otimes TM \to TM; TM^f \subset TM; g:S^2(TM)\to O_M),\end{equation}
where \begin{itemize}
\item $M$ is a manifold; 
\item $\circ$ is a multiplication operation on the tangent space $TM$ to $M$; 
\item $TM^f $ is a subspace of $TM$
 of linear spaces of pairwise commuting vector fields; 
 \item $g$ is an even symmetric pairing (i.e. here a non-degenerate Riemannian metric) with $O_M$ a so-called structure sheaf.
 \end{itemize}
 
 \smallskip 
 
A Frobenius manifold is a manifold equipped with an associativity axiom and potentiality axiom (see~\cite{Man99} p.19). These properties can be summarised as having a family of local potentials $\Phi$ (sections of $O_M$) such that for any local flat tangent fields $X,Y, Z$ we have :
\[g(X\circ Y, Z)=g(X, Y\circ Z)=(XYZ)\Phi.\]

If such a structure exists, then (super)commutativity and associativity of the multiplication $\circ$  follows naturally, and we say that the family defines a Frobenius manifold.

\subsection*{Alternative description of the associativity condition}
It is possible to encapsulate the associativity relation of the Frobenius' structure within the Witten--Dijkgraaf--Verlinde--Verlinde (WDVV) highly non-linear PDE system: 
\begin{equation}\label{E:WDVV}\forall a,b,c,d: \underset{ef}{\sum}\varPhi_{abe}g^{ef}\varPhi_{fcd}=(-1)^{a(b+c)}\underset{ef}{\sum}\varPhi_{bce}g^{ef}\varPhi_{fad}.\end{equation}
Notice that the (WDVV) system expresses {\it a flatness condition of the manifold} (i.e. vanishing of the curvature),

\subsection*{Frobenius structure and flatness}
The structure of Frobenius manifold is deeply related to the notion of flat connections, as we well see.
Let us introduce the Levi--Civita connection $\nabla_0 : \cT_M \to \Omega^1_M\otimes_{\cO_M} \cT_M$,
uniquely defined by the horizontality of $\cT^f_M$. 
 \vspace{3pt}

 The Levi--Civita connection $\nabla_0$ can be further on deformed  into a pencil of connections $\nabla_{\lambda}$
depending on an even parameter $\lambda$. The respective covariant derivative is given by the following equation:
\[
\nabla_{\lambda ,X}(Y) := \nabla_{0,X}(Y)+ \lambda X\circ Y.
\]
This pencil  of connections is referred to as the {\it structure connection} of our pre--Frobenius manifold.
Now, by~\cite{Man99}, Theorem 1.4, one can prove that:
\begin{theorem}[\cite{Man99}, Theorem 1.4]  $\nabla_{\lambda,X}(Y)$ is flat if and only if the manifold is Frobenius.\end{theorem}

\subsection{The fourth Frobenius manifold and information geometry}
We now move towards the definition of the  fourth Frobenius manifold.  Consider the family data $(M,g,\nabla)$, where $M$ is a differentiable manifold of probability distributions related to exponential families, equipped with a positive-definite metric tensor $g$ and a torsion-free arbitrary affine connection $\nabla$ for which $(\nabla Xg)(Y,Z)$ is symmetric for $X$ and $Y$.

The linear connection $\nabla^*$ defined by
\[ Xg(Y, Z) = g(\nabla_X Y, Z) + g(Y, \nabla^*_X Z),\]
is called the conjugate (dual) connection of $\nabla$ with respect to $g$.
We focus our attention on the case where we have an $F$-manifold structure\, \cite{CoMa} i.e. when the pair of conjugate connections $\nabla$ and $\nabla^*$ are flat. 

Another terminology used to designate this type of manifolds is ``statistical manifold''. Equivalently, a statistical manifold is a pseudo Riemannian manifold $(M,g)$ equipped with a pair of symmetric conjugate connections. The triple $(g,\nabla,\nabla_*)$ is called a statistical structure on $M$. 

\subsection{Paracomplex structures}
In~\cite{CoMa} it turned out that a very important algebra to consider, in the context of statistical manifolds, is the algebra of paracomplex numbers. This real algebra is the key towards understanding many statements concerning the fourth Frobenius manifold, in a concise and more algebraic way. 

\smallskip 

Let $\fA$ be the unital and bi-dimensional algebra, generated by $1$ and $\e$ where $\e^{2}=1$, verifying the following relations:
 \[e_{i}\cdot e_{j}=\sum_{k}C^{k}_{ij}e_{k}\quad \text{with}\quad C^{k}_{ij}=C^{k}_{ji}.\]
 The algebra $\frak{A}$ is known as the spin factor algebra or the algebra of paracomplex numbers. 

\vskip.2cm 

The structure constants $C^{k}_{ij}$ are:
\[C_{11}^{1}=C_{12}^2=C_{22}^1=\, 1, \]
the other structure constants are null.

Let us change the basis such that  the new generators are defined by:
 \[e_{-}= \frac{1-\e}{2},\quad e_{+}=\frac{1+\e}{2}.\]
These generators have the following relations:
 \[ e_{-}\circ e_{-}=e_{-},\  e_{+}\circ e_{+}=e_{+},  e_{-}\circ e_{+}=0,\]
 \[e_{-}+ e_{+}=1,\ e_{-}- e_{+}= \e.\]
we call this new basis a canonical basis. Notice that this new basis highlights the existence of a pair of idempotents i.e. $e_{-}^2=e_{-}$ and $e_{+}^2=e_{+}$.

{ \bf Remark} This semi-simple algebra is isomorphic to $\mathbb{R}\oplus \mathbb{R}$.
As a set, it can be identified to $\mathbb{R}^{2}$ but not as an algebra. 

\section{The algebraic approach}
It has been shown in \cite{CoMa} that there exists a tight relation between the manifold of probability distributions (related to exponential families) and a certain class of Vinberg cones. The $n$-dimensional Vinberg cones are in bijection with $n$-Jordan algebras~\cite{Vi}. 

For the convenience of the reader we recall some known results concerning Jordan algebras, Peirce decomposition and Peirce mirror/ reflection in section~\ref{S:P}. For more information, see \cite{MC04,Pei}).

In subsection~\ref{S:V} we comment on the relation between Vinberg cones and the manifold of probability distributions.

\subsection{A brief survey on Peirce's reflections}\label{S:P}
 Peirce's decomposition arises in the context of associative rings and also Jordan algebras (see section 6.1 in \cite{MC04}). A Peirce decomposition of an associative algebra (or ring) is a decomposition of this algebra (or ring) as a sum of eigenspaces of commuting idempotent elements.

Given an idempotent of an algebra $A$ (i.e. an element $a\in A$ that verifies $a^2=a$), there exist right, left and two-sided Peirce decompositions, which are respectively defined as:

\[A=Aa+A(1-a),\]
\[A=aA+(1-a)A,\]
\[A=aAa+aA(1-a)+(1-a)Aa+(1-a)A(1-a).\]
 The latter writes $A$ as the direct sum of $aAa, aA(1-a), (1-a)Aa$ and $(1-a)A(1-a).$

In a more general setting, for associative rings, a decomposition  $\sum_i a_i=1$ of the unit into mutually
{\it orthogonal idempotent elements}, (i.e. verifying the condition $a_ia_j=a_ja_i=0$) produces a decomposition of the algebra into Peirce spaces: \[A=\sum_{i,j}a_iAa_j.\]  

\begin{remark}We can recover the structure of the entire algebra by analysing the structure of the individual pieces and how they are reassembled to form $A$.\end{remark}

\smallskip 

Jordan algebras have similar Peirce decompositions. We call Jordan algebra a commutative algebra $J$ where the following Jordan identity holds for all $x,y$ in $J$:
\[x\bullet (x^2\bullet y)=x^2\bullet (x \bullet y).\]

\begin{remark}Note that any associative algebra $A$ over the reals gives rise to a Jordan algebra $J$ under
quasi-multiplication: the product $x \bullet y := \frac{1}{2} (xy + yx)$ is clearly commutative, and satisfies the Jordan identity. 
\end{remark}

The only difference with the case above being that the “off- diagonal” Peirce spaces $J_{ij}$ were $(i\neq j)$ behave like the sum 
$A_{ij} + A_{ji}$ of two associative Peirce spaces, so that we have symmetry $J_{ij} = J_{ji}$. For $n = 2$ the decomposition is determined by a single idempotent $a = a_1$ with its complement $a' := 1 - a$, and can be described more simply.

In particular, we have te following  Peirce Decomposition Theorem.
\begin{theorem}[Peirce decomposition]
An idempotent $a$ in a Jordan algebra $J$ determines a Peirce decomposition of the algebra into the direct sum of Peirce subspaces:
\[ J = J_2\oplus J_1\oplus J_0\quad J_k:=J_k(a):=\{x\, |\, V_ax=k x\},\quad \text{($V$ is a linear space)}.\]
The diagonal Peirce subalgebras $J_2 = U_a J,$ $J_0 = U_{1-a}J$ are principal inner ideals.
\end{theorem}
More generally, a decomposition $1 = a_1 + \dots + a_n$ of the unit of a Jordan algebra $J$ into a sum of $n$ supplementary orthogonal idempotents leads to the following Peirce decomposition of the algebra into the direct sum of Peirce subspaces:
\[J=\oplus_{i\leq j} J_{ij}.\]
Note that the diagonal Peirce subalgebras are the inner ideals

\[J_{ii}:=U_{a_i}(J)=\{x\in J\, |\, a_i \bullet x =x\}\]
and the off-diagonal Peirce spaces are
\[ J_{ij}=J_{ji}:=U_{a_i,a_j}(J)=\{x\in J\, |\, a_i \bullet x=a_j\bullet x=\frac{1}{2}x\}\quad (i\neq j).\]
Because of the symmetry $J_{ij} = J_{ji}$ the multiplication rules for these subspaces are a bit less uniform to describe, but basically the product of two Peirce spaces vanishes unless two of the indices can be linked.

The notion of Peirce reflection (or Peirce mirror) arises very naturally. Let $J$ be a unital Jordan algebra. Consider a pair of orthogonal idempotents $a, a'$ and put $u=a-a'$. Then, it is easy to check that $u$ is an involution. Now, 
if $u$ is an involution in a unital Jordan algebra $J$, then the involution $U_u$ is an involutionary automorphism of $J$, i.e. $U_u^2 = 1_J$ (see lemma 6.10 in \cite{MC04}). 
Then the involution defined as $U_u=E_2-E_1+E_0$ has a description as a Peirce reflection in terms of the Peirce projections $E_i(a)$ (see p.238 \cite{MC04} for more details).

We will use this Peirce decomposition construction and reflection on the Frobenius algebra $\fA$ of paracomplex numbers and then study the consequences for the module over this algebra.

\subsection{Paracomplex Vinberg cone, Chentsov cone and Jordan algebra}\label{S:V}
Let $\mathcal{W}$ be a linear space of signed measures with bounded variations, vanishing on an ideal $\mathcal{I}$ of the $\sigma$-algebra $\mathcal{F}$. Let $\cC$ be a cone in $\mathcal{W}$ of (strictly) positive measures on the space $(X,\cF)$, vanishing only on an ideal $\cI$ of the $\sigma$--algebra $\cF.$
 
As it was shown in\, \cite{CoMa}, the cone of positive measures of bounded variations belongs to one of the five Vinberg cones, i.e. the cone defined over the algebra of paracomplex numbers. More precisely we have the following statement. 

\begin{theorem}\label{T:+++}
The positive cone $\cC$, defined above, is a paracomplex Vinberg cone. 
\end{theorem}

Recall that a cone $V\subset R$ is a non--empty subset, closed with respect to addition and multiplication by positive reals. A convex cone $V$  in a vector space $R$ with an inner product has a dual cone $V^*=\{a\in R:\,  \forall b\in V,\, \langle a,b\rangle>0 \}$. The cone is self-dual when  $V=V^*$. It is homogeneous when to any points $a,b\in V$ there is a real linear transformation $T:V\to V$ that restricts to a bijection $V\to V$ and satisfies $T(a)=b$.
Moreover, the closure of $V$ should not contain a real linear subspace of positive dimension. A Vinberg cone is a convex, self dual, homogeneous cone.

An important connection between Jordan algebras and Vinberg $n$-cone is given by the fact that {\it there is a bijection between an $n$-Vinberg cone and a (semi-simple) Jordan $n$-algebra.} 

There exist five main classes of  Vinberg cones. Any other Vinberg cone can be obtained by a linear combination of one of the five cones of the list below. 
\begin{proposition}~\label{P:Vclass}
Each irreducible homogeneous self--dual (Vinberg) cone belongs to one
of the following classes:
\vspace{3pt}

\begin{enumerate}
\item The cone $M_+(n, \R)$ of $n \times n$ real positive matrices.
\vspace{3pt}

\item The cone $M_{+}(n, \C)$ of $n \times n$ complex positive matrices.
\vspace{3pt}

\item The cone $M_{+}(n, \H)$ of $n \times n$ quaternionic positive matrices.
\vspace{3pt}

\item The cone $M_{+}(3, \O)$ of $3 \times 3$ positive matrices whose elements are in $\O$, the Cayley algebra (also known as the Octonionic algebra).
\vspace{3pt}

\item The cone $M_{+}(n, \fA)$ of $n \times n$ paracomplex positive matrices. 
\end{enumerate}
\end{proposition}

\vspace{3pt}
\begin{remark} A matrix is positive if it is self-adjoint and its eigenvalues are positive.\end{remark}

Note that the bijection follows from the fact (Theorem 5.2 \cite{CoMa}):
\begin{enumerate}
\item The list of algebras in Proposition~\ref{P:Vclass} coincides with the list
of all irreducible finite dimensional formally real algebras.
\vspace{3pt}

\item The irreducible homogeneous self-dual cone associated with such an algebra
$\fM$ is the set of positive elements of a Jordan algebra, i.e. elements represented by positive matrices.
\end{enumerate}

  \smallskip 

From now on, we now expose several properties, which are the building blocks of the proof of the main theorem.

\subsection{Building blocks of the main theorem}

We will prove step by step the statement of the Main theorem, by splitting it into a collection of smaller statements.
\begin{lemma}\label{L:2}
The $n$-dimensional fourth Frobenius manifold is identified to a linear space being the (real) realisation of a module over the spin factor algebra. 
\end{lemma}

This lemma implies a family of geometrical results, which we enumerate below.
\begin{proposition}[Paracomplex spaces]\label{P:1}
The  fourth Frobenius manifold is a paracomplex space. 

\end{proposition}

\begin{theorem}[Paracomplex manifold]\label{T:paramfd}
The fourth Frobenius manifold is identified to a projective paracomplex manifold. 
\end{theorem}

\begin{lemma}\label{c:manif}
Let $\fA$ be the spin factor algebra. Let us consider the pair of ideals in $\fA$, generated by the idempotents. Then, the fourth Frobenius manifold is decomposed into a direct sum of a pair of totally geodesic submanifolds, being real realisations of a cartesian product of modules over these ideals.\end{lemma}

The latter lemma can be proved by computing the idempotents of  $\fA$ and relying on Shurigin's statements \cite{Shu}.

\v Centsov\, \cite{Ce3} has shown that the totally geodesic submanifolds turn out to be the exponential families of probability distributions with local parametrization.  Therefore, we can obtain the following:

\begin{lemma}\label{T:3}
\v Centsov's statistical manifold has a pair of real flat connections. 
\end{lemma}
This follows from the calculation in \cite{BuNen1,BuNen2}.

\begin{proposition}\label{P:Pierce}
The fourth Frobenius manifold is decomposed into pseudo-Riemannian submanifolds, which are symmetric to each other with respect to the Peirce mirror.
\end{proposition}

In order to prove these statements, we recall the tools from paracomplex geometry, in the following subsections. 
\section{Proofs}
Let us briefly introduce the fundamental tool proving the Lemma\, \ref{L:2} and Proposition\, \ref{P:1}: the construction of modules over the algebra of paracomplex numbers and its real realisation. 

\subsection{Modules over spin factor algebra}

We construct the module over the spin factor algebra. Let $\frak{A}$ be the spin factor algebra. Let us construct an $m$-module over the spin factor algebra $\frak{M}^{m}(\frak{A})$. The affine representation of the algebra $\fA$, or free module $\fA E^{m}$, admits a real interpretation in the real linear space $E^{2m}$(\cite{Roz97}, section 2.1.2).

Let $E^{2m}$ be a $2m$-dimensional real linear space. {\it A paracomplex structure} on $E^{2m}$ is an endomorphism $\fK: E^{2m} \to E^{2m}$ such that $\fK^2=I$. The eigenspaces $E^{m}_+, E^{m}_-$ of $\fK$ with eigenvalues $1,-1$ respectively, have the same dimension.

The pair $(E^{2m},\fK)$ will be called a {\it paracomplex vector space.} We define {\it the paracomplexification of $E^{2m}$} as $E^{2m}_\fA = E^{2m} \otimes_{\R} \fA$ and we extend $\fK$ to a $\fA$-linear endomorphism $\fK$ of $E^{2m}_\fA$.

\begin{lemma}

Let $E^{2m}_\fA = E^{2m} \otimes_{\R} \fA$ be endowed with an involutive $\fA$-linear endomorphism $\fK$ of $E^{2m}_\fA$. Then,  the space $E^{2m}_\fA$  is decomposed into the direct sum of a pair of $m$-dimensional subspaces $E^{m}_{+}$ and $E^{m}_{-}$ such that: \[E^{2m}_\fA=E^{m}_{+}\oplus E^{m}_{-},\] verifying:
\[
E^{m}_{+} = \{v\in E^{2m}_\fA \, |\, \fK v=\e v\}=\{v+\e\fK v\, |\, v \in E^{2m}_\fA\},
\]
\[
E^{m}_{-} = \{v\in E^{2m}_\fA |\,  \fK v= -\e v\}=\{v-\e\fK v\, |\, v\in E^{2m}_\fA\}.
\]

\end{lemma}

{\bf Proof} This statement and its proof can be found in the literature. For example, see \cite{Roz97}.

\begin{remark}
From this, it follows that any hypersurface $x^i=0$ of $E^{2m}_\fA$ is decomposed into two hypersurfaces $x_{+}^{i}=0$ and $x_{-}^{i}=0$ respectively in $E^{m}_{+}$ and $E^{m}_{-}$. Similarly, the coordinates of any point of the space $E^{2m}$ can be given  by \[x^i=x_{+}^ie_{+}+x_{-}^{i}e_{-}.\]
  
Furthermore, consider a vector ${\bf{X}}(x^i)$ at the coordinate $x^i$ in the space $E^{2m}_\fA$. Then, due to the splitting we have that ${\bf{X}}(x^i)={\bf{X}}(x_{+}^i)\oplus{\bf{X}}(x_{-}^i)$. Therefore, the vector $\bf{X}$ splits into ${\bf{X}}={\bf{X}}_{+}\oplus{\bf{X}}_{-}$ and ${\bf{X}}^2$ is given by $\sum_{i} x_{+}^ix_{-}^i$.
\end{remark}
\subsection*{Proofs of Lemma~\ref{L:2} and Proposition~\ref{P:1}}
\begin{proof}[of Lemma \ref{L:2}]
The tangent space is identified with the space of signed measures with bounded variations, being absolutely continuous with respect to $P_{\theta}$ and such that $\mu(\Omega)=\int_{\Omega}u dP_{\theta}=0$, with $u\in L^2(\Omega, P_{\theta})$. Such a measure can be split into two positive measures $\mu^+$ and $\mu^-$ such that: $\mu=\mu^+-\mu^-$ and $|\mu|=\mu^++\mu^-<\infty$. This forms a real affine space. 
Using the Norden--Shirokov  construction in \cite{No58,Sh02}, leads to the existence of a module over an algebra of rank 2. Since this algebra is not complex, there remain only possibilities: the algebra of paracomplex numbers or the dual numbers. By contradiction, we eliminate the case of dual numbers a (i.e. it is clear that there do not exist any generators such that their square gives 0).
\end{proof}

\begin{proposition}[{\it The rank 2} Lemma,\, \cite{CoMa}]\label{P:rank2}
Consider an affine, symmetric space over a Jordan algebra. There exists exactly two affine and flat connections on this space if and only if the algebra is of rank 2, and generated by $\{1, \e\}$ with $\e^2= 1$ or $-1$.
\end{proposition}

\begin{proof}[of Proposition \ref{P:1}]
Let $\cC\subset \mathcal{W}$ be the cone of positive measures.  Let $\cH\cap \cC\neq \emptyset$. By Theorem \ref{T:+++}, $\cC$ is a cone over $\fA$.  So, the fourth Frobenius manifold which corresponds to a section of the cone $\cC$ inherits the paracomplex structure. 
We have shown in Lemma~\ref{L:2} that the tangent space to this $F$-manifold is identified to a module over an algebra of rank 2. Applying the proposition \ref{P:rank2}, this algebra is a paracomplex one. \end{proof}

\subsection{Paracomplex manifolds}
  
In this subsection, we expose paracomplex manifolds and their properties. This will be used to prove theorem \ref{T:paramfd}.
Let $y=f(x)$ be a (analytic) function, whose domain and range belong to a commutative algebra (i.e. $C_{jk}^h=C_{kj}^h$). We put $x=\sum_ix_ie_i,$ $y=\sum_iy_ie_i.$
From the generalized Cauchy--Riemann we have the following:
\[\sum_h\frac{\partial y_i}{\partial x_h}C_{jk}^h=\sum_h\frac{\partial y_h}{\partial x_i}C_{hk}^j,\]
where $C_{jk}^h$ are the constant structures. 

\vskip.2cm

 A {\it paracomplex manifold} is a real manifold $M$ endowed with a paracomplex structure $\fK$ that admits an atlas of paraholomorphic coordinates (which are functions with values in the algebra $\fA = \R + \e\R$ defined above), such that the transition functions are paraholomorphic.

Explicitly, this means the existence of local coordinates $(z_+^\alpha, z_-^\alpha),\, \alpha = 1\dots, m$  such that
paracomplex decomposition of the local tangent fields is of the form
\[
T^{+}M=span \left\{ \frac{\partial}{\partial z_{+}^{\alpha}},\, \alpha =1,...,m\right\} ,
\]
\[
T^{-}M=span \left\{\frac{\partial}{\partial z_{-}^{\alpha}}\, ,\, \alpha =1,...,m\right\} .
\]
Such coordinates are called {\it adapted coordinates} for the paracomplex structure $\fK$.

By abuse of notation, we write $\partial_z$ instead of $\frac{\partial}{\partial z^{\alpha}}$.

  We associate with any adapted coordinate system $(z_{+}^{\alpha}, z_{-}^{\alpha})$ a paraholomorphic coordinate system $z^{\alpha}$ by 
\[
z^\alpha\, =\, \frac{z_{+}^{\alpha}+z_{-}^{\alpha}}{2} +\e\frac{z_{+}^{\alpha}-z_{-}^{\alpha}}{2}, \alpha=1,...,m .
\]

We define the paracomplex tangent bundle as the $\R$-tensor product $T^\fA M = TM \otimes \fA$ and we extend the endomorphism $\fK$ to a $\fA$-linear endomorphism of $T^\fA M$. For any $p \in M$, we have the following decomposition of $T_{p}^\fA M$:
\[
T_p^\fA M=T_p^{1,0}M \oplus T_p^{0,1}M\,
\]
where 
\[
T_p^{1,0}M = \{v\in T_p^\fA M | \fK v=\e v\}=\{v+\e \fK v| v \in E^{2m}\} ,
\] 
\[
T_p^{0,1}M = \{v\in T_p^\fA M | \fK v= -\e v\}=\{v-\e \fK v|v\in E^{2m}\}
\]
are the eigenspaces of $\mathfrak{K}$ with eigenvalues $\pm \e$.
The following paracomplex vectors 
\[
\frac{\partial}{\partial z_{+}^{\alpha}}=\frac{1}{2}\left(\frac{\partial}{\partial x^{\alpha}} + \e\frac{\partial}{\partial y^{\alpha}}\right),\quad \frac{\partial}{\partial{z}_{-}^{\alpha}}=\frac{1}{2}\left(\frac{\partial}{\partial x^{\alpha}} - \e\frac{\partial}{\partial y^{\alpha}}\right)
\]
form a basis of the spaces $T_p^{1,0}M$ and $T_p^{0,1}M$.
\subsection*{Proofs of Theorem~\ref{T:paramfd} and Proposition~\ref{P:Pierce}}
\begin{proof}[{\it of Theorem \ref{T:paramfd}}]
We rely on the above subsection on paracomplex manifolds. Consider an affine and symmetric space over a Jordan algebra $A$ and consider $S$ as defined in the Appendix. From \cite{No58,Am85,Am97,BuNen1,BuNen2}, it is shown that there exist 2 connections. Applying Proposition\, \ref{P:rank2}, we have that the algebra is of rank 2, and generated by $\{1, \e\}$ with $\e^2= 1$ or $-1$. These flat affine connections are constructed from a field of objects, having the components: 
\[\Gamma_{jk}^{i}=\Gamma_{jk}^{i \alpha}e_{\alpha} \in A.\]

Suppose that $\bv^{i}=\bv^{(i,\alpha)}e_{\alpha}$ are quantities from the algebra corresponding to a tangent vector $\bv$ to $S$. 
Then, from the following condition
\[d\bv^{i} + \Gamma^{i}_{j k}\bv^{j} d\bx^{k}= 0,\]
we can define an affine connection equipped with the following components:
\[\Gamma^{(i,\alpha)}_{(j,\beta)(k,\gamma)}=\Gamma_{j k}^{i s}C_{s \beta}^\delta C^{\alpha}_{\delta \gamma} ,\]
where the $C_{ \beta \gamma}^{\alpha}$ are structure constants of algebra $A$, with respect to the local adapted coordinates $x^{(\alpha,i)}$. 
Now, these objects are indexed by the number of generators of the algebra $A$. There exist 2 connections. So, it implies that $s\in\{1,2\}$ and $A$ is of rank 2, generated by $\{1,\e\, |\, \e^2=\pm1\}$. Since, we cannot be in a complex framework, we have an algebra of paracomplex numbers.
\end{proof}
 
\begin{proof}[{\it  of Proposition \ref{P:Pierce}}]
By theorem\, \ref{T:paramfd} we know that $S$ (as defined in the Appendix) is a paracomplex manifold. 
Applying Peirce's decomposition theorem, we can decompose the algebra of paracomplex numbers using the pair of ideals which are orthogonal idempotents.
This implies two things. Firstly, it implies that the manifold can be decomposed into a pair of submanifolds, being respectively the realisation of the pair of modules defined over the pair of orthogonal, idempotent ideals of $\fA$. Following Shurygin's statements in~\cite{Shu}, we have that these submanifolds have flat connections and are totally geodesic.  

Secondly, relying on \cite{MC04}, Section. 8.1, we have fulfilled the necessary conditions for the existence of a Peirce reflection. The construction is as follows. One submanifold corresponding to a module over one idempotent ideal is the reflection of the other submanifold, corresponding to the module over the orthogonal idempotent, under Peirce's mirror. \end{proof}
\section{Conclusion}
 We can now proceed to discussing the final steps proving the main theorem presented at the beginning of the paper. 
The results are summarised in the table of corollary 1. The classical setting for statistical manifolds, under consideration in this paper is given in the Appendix.\begin{theorem}[{\bf Main Theorem}]
 The statistical manifold related to exponential families satisfies the following properties.
\begin{enumerate}
\item It is identified to a projective manifold, which is defined over a rank two Frobenius algebra $\fA$. This Frobenius algebra is a group ring of a finite dimension 2 group, over the field of real numbers, generated by \[\langle 1,\ \e\, |\, \e^{2}=1\rangle.\] 
\item There exists a decomposition of the fourth $F$-manifold into a pair of totally geodesic submanifolds\footnote{Let $M$ be a manifold and $g$ a Riemannian metric.  A submanifold $N$ of a Riemannian manifold $(M,g)$ is called {\it totally geodesic} if any geodesic on the submanifold $N$ with its induced Riemannian metric $g$ is also a geodesic on the Riemannian manifold $(M,g)$. For more details, we refer to \cite{KoNo}.
}, containing  a pair of flat connections. These correspond to the realisations of the module over the maximal ideal of $\fA$ and its annihilator.
\item This pair of pseudo-Riemannian submanifolds of the fourth $F$-manifold are symmetric to each other, with respect to a Peirce reflection.
\end{enumerate}
\end{theorem}

\begin{proof}
The part (1) is proved by applying lemma \ref{L:2} and theorem\, \ref{T:paramfd}. It turns out that the statistical manifold related to exponential families is a paracomplex manifold. 

The part (2) follows from lemma \ref{c:manif} and theorem \ref{T:3}. 

The part (3) follows from proposition \ref{P:Pierce}.
\end{proof}

\appendix
\section{Appendix: Statistical manifolds}
\subsection*{Probabilistic setup}
Let $(\Omega,\mathcal{F})$ be a measure space, where $\mathcal{F}$ denotes the $\sigma$-algebra of elements of $\Omega$. 
We consider a family of parametric probabilities  $\frak{S}$ on the measure space $(\Omega,\mathcal{F})$. We ask all parametric probabilities of $\frak{S}$ to be {\it absolutely continuous} w.r to a $\sigma$-finite measure $\lambda $ i.e. a measure $P \in (\frak{S},\mathcal{F})$ is absolutely continuous w.r to $\lambda$ if for every measurable set $A\subset \mathcal{F}$:
\[\lambda(A)=0 \Rightarrow P(A)=0, \quad \forall A\subset \mathcal{F}, \]which we denote by $P\ll \lambda$. Note that
$P \ll \lambda$  does not imply that $\lambda \ll P$.

If $\lambda$ is  positive and $\sigma$-finite there exists a measurable function $\rho$ called density of the measure $P$ w.r to the measure $\lambda$, denoted 
\[\rho= \frac{dP}{d\lambda},\quad \quad P(A)=\int_{A}\rho d\lambda,\  \forall A\subset \mathcal{F}. \]
We denote by $S$ the associated family of probability densities of the parametric probabilities. We limit ourselves to the case where $S$ is a smooth topological manifold. 

\[S= \left\{ \rho_{\theta} \in L^1(\Omega, \lambda),\  \theta = \{\theta_1,\dots \theta_n\}; \ \rho>0\  \lambda - a.e., \ \int_{\Omega}\rho d\lambda=1\right\}.\]

This generates the space of probability measures absolutely continuous  with respect to the measure $\lambda$, i.e.
$P_{\theta}(A)=\int_A  \rho_{\theta}d\lambda$ where $A\subset \mathcal{F}$.

We construct its tangent space as follows. Let $u\in L^{2}(\Omega, P_{\theta})$ be a tangent vector to $S$ at the point $P_{\theta}$. 
\[T_{\theta}=\left\{ u\in L^{2}(\Omega, P_{\theta}); \mathbb{E}_{P_{\theta}}[u]=0, u=\sum_{i=1}^{d}u^i\partial_{i}\ell_{\theta}         \right\},\] 
where $\mathbb{E}_{P_{\theta}}[u]$ is the expectation value, w.r. to the probability $P_{\theta}$.

The tangent space of $S$  is isomorphic to the $n$-dimensional linear space generated by the centred random variables (also known as score vector) $\{ \partial_{i}\ell_\theta\}_{i=1}^n$, where $\ell_{\theta} = \ln \rho_{\theta}.$ 

\begin{remark}
The elements of $T_\theta$ generate the family  of signed measures with bounded variations (i.e. signed measures whose total variation $ \Vert \mu \Vert =|\mu |(X) $ is bounded, vanishing only on an ideal $\mathcal{I}$ of the $\sigma$-algebra $\mathcal{F}$) which are absolutely continuous  with respect to  $P_\theta$ and such that $\int _{\Omega} u dP_{\theta}=0$. This forms a real affine space. \end{remark}
 In 1945, Rao \cite{Rao45} introduced the Riemannian metric on a statistical manifold, using the Fischer information matrix. The statistical manifold forms a (pseudo)-Riemannian manifold. 

In the basis, where $\{ \partial_{i}\ell_\theta\}, \{i=1,\dots,n\}$ where $\ell_{\theta} = \ln \rho_{\theta},$  the Fisher metric are just the {\it covariance matrix} of the score vector. In particular, we have:

\[  g_{i,j}(\theta)=\mathbb{E}_{P_{\theta}}[\partial_i\ell_{\theta}\partial_{j}\ell_{\theta}]\]
\[  g^{i,j}(\theta)=\mathbb{E}_{P_{\theta}}[a^{i}_{\theta}a^{j}_{\theta}], \]

where  $\{a^{i}\}$ form a dual basis to $\{\partial_j\ell_{\theta}\}$:
\[a^{i}_{\theta}(\partial_j\ell_{\theta})=\mathbb{E}_{P_{\theta}}[a^{i}_{\theta}\partial_j\ell_{\theta}]=\delta^i_{j}\]
with
\[\mathbb{E}_{P_{\theta}}[a^{i}_{\theta}]=0.\]

\end{document}